\documentclass [journal,onecolumn,11pt]{IEEEtran}
\usepackage{amsfonts,amsmath,amssymb}
\usepackage{indentfirst, setspace}
\usepackage{url,float}
\usepackage{color}
\usepackage[a4paper,ignoreall]{geometry}
\usepackage{longtable}
\usepackage{array}
\usepackage{graphicx}
\usepackage[a4paper,ignoreall]{geometry}
\usepackage{multicol}
\usepackage{stfloats}
\usepackage{enumerate}
\usepackage{cite}
\usepackage{amsthm}
\usepackage{multirow}
\usepackage{amssymb}
\usepackage{subeqnarray}
\usepackage{amssymb}
\usepackage{cases}
\usepackage[square, comma, sort&compress, numbers]{natbib}

\def\qu#1 {\fbox {\footnote {\ }}\ \footnotetext { From Qu: {\color{red}#1}}}
\def\hqu#1 {}
\def\kq#1 {\fbox {\footnote {\ }}\ \footnotetext { From KangQuan: {\color{blue}#1}}}
\def\hkq#1 {}

\usepackage{extarrows}
\usepackage{titlesec}
\usepackage{amssymb}
\usepackage{amsmath,amsthm,amssymb,color}
\usepackage[pdftex]{hyperref}
\usepackage{graphicx}
\usepackage[varg]{txfonts}
\usepackage{multirow}
\date{}
\usepackage{url}

\newtheorem{conjecture}{Conjecture}

\newtheorem{lem}{Lemma}[section]
\newtheorem{thm}[lem]{Theorem}
\newtheorem{prop}[lem]{Proposition}

\newtheorem{defn}[lem]{Definition}

\begin{document}
\title{Further results on  permutation pentanomials over ${\mathbb F}_{q^3}$  in
	characteristic two }
\author{Tongliang Zhang\thanks{
		$ ^{\ast} $ Corresponding author.
		\newline \indent T. Zhang is with the College of Science, North China University of Science and Technology, Tangshan, Hebei, 063000, China, Hebei Key Laboratory of Data Science and Application, North China University
		of Science and Technology, Tangshan, Hebei, 063000, China;
		L. Zheng and H. Wang are with the School of Mathematics and Physics, University of South China, Hengyang, Hunan, 421001, China,~(E-mail:~zhenglijing817@163.com); J. Peng is with the Mathematics and Science College, Shanghai Normal University, Shanghai 200234, China;
		 Y. Li is with the Institute of Statistics and Applied Mathematics,
		Anhui University of Finance and Economics, Bengbu, Anhui, 233030, China.}, Lijing Zheng$ ^{\ast}$, Hengtai Wang, Jie Peng, Yanjun Li}
\maketitle

\maketitle
\begin{abstract}
 Let  $q=2^m.$     In a recent paper \cite{Zhang3}, Zhang and Zheng investigated several classes of permutation pentanomials  of the form $\epsilon_0x^{d_0}+L(\epsilon_{1}x^{d_1}+\epsilon_{2}x^{d_2})$ over ${\mathbb F}_{q^3}~(d_0=1,2,4)$
 from some certain  linearized polynomial $L(x)$ by using multivariate method and some   techniques  to determine the number of the solutions of some  equations. They proposed an open problem that there are still   some   permutation pentanomials  of that  form  have not been proven. In this paper, inspired by the idea of \cite{LiKK}, we further characterize the  
 permutation property of the  pentanomials with the above form  over ${\mathbb F}_{q^3}~(d_0=1,2,4)$.  The techniques in this paper would be useful to investigate more new
 classes of permutation pentanomials.


 \end{abstract}

\section{Introduction}

Let ${\mathbb F}_q$ be the finite field with $q$ elements. A polynomial $f\in \mathbb{F}_{q}[x]$ is called a \emph{permutation polynomial} (PP) over $\mathbb{F}_q$, if its associated polynomial mapping $f:c\mapsto f(c)$ from  ${\mathbb F}_q$ to itself is a bijection.
Permutation polynomials over finite fields
have applications in coding theory, cryptography, combinatorial designs, and other areas of mathematics and engineering. 
Information about
properties, constructions, and applications of permutation polynomials may be found in 
  the survey papers \cite{Hou3,Li3,Wang}.

Permutation polynomials with few terms have attracted
much attention for their simple algebraic forms (see for instance \cite{Ba2, Ding,Gupta,Hou1,Hou2,Hou3, Hou4,Li1,Li2,LiK,Pang,Peng,Zhang1,Zheng1,Zheng2,Zheng3}). 
The construction of permutation  pentanomials is an interesting
research problem. Let $q=2^m$.
Xu et al.  \cite{Xu} constructed some classes of  permutation pentanomials with the form $x^rh(x^{q-1})$ over
$ \mathbb{F}_{q^{2}}$.
Liu et al. \cite{Liu} studied the permutation property of pentanomials with the form $x^rh(x^{p^m-1})$ over
 $ \mathbb{F}_{p^{2m}}$ for $p=2,3.$
	Very recently,  Zhang et al. \cite{Zhang2} 
	investigated  all  new permutation pentanomials  of the form $$f(x)=x^t+x^{r_1(q-1)+t}+x^{r_2(q-1)+t}+x^{r_3(q-1)+t}+x^{r_4(q-1)+t}$$   with ${\rm gcd}(x^{r_4}+x^{r_3}+x^{r_2}+x^{r_1}+1,x^t+x^{t-r_1}+x^{t-r_2}+x^{t-r_3}+x^{t-r_4})=1, 4\leq t<100$ over $ \mathbb{F}_{q^2}$.  Utilizing  a well-known lemma, they transformed the problem of proving that $ f $ is a permutation of $ \mathbb{F}_{q^{2}} $ into
	showing that some specific equation over $ \mathbb{F}_{q} $ (related to $ f $) has no   certain solution  in $ \mathbb{F}_{q} $.

	
	The construction of permutation  polynomials  over  $\mathbb{F}_{q^3}$ are also of great interest recently. 
	In   a recent work, Wang et al.  \cite{Wang-Zhang-Zha} and Li et al. \cite{LiK} respectively proposed two classes of permutation trinomials of  the form $x^{d_1}+L(x^{d_2})$ with  the linearized polynomials $L(x)=x+x^q$ over $\mathbb{F}_{q^3}$ by applying the multivariate method.
	In \cite{Gupta1}, Gupta et al. constructed three classes of permutation trinomials over $\mathbb{F}_{q^3}$. In that paper, they also confirmed a conjecture proposed by Gong et al. \cite{Gong} about permutation trinomials of the form $x^{q+1}+L(x)\in\mathbb{F}_{q^3}[x]$.
	Recently, Pang et al.  \cite{Pang} investigated   permutation trinomials over $\mathbb{F}_{q^3}$ of the  form $x^d+L(x^s)$ with certain integers
	$d, s$ and  the linearized polynomials $L(x)=ax+bx^q\in\mathbb{F}_{q^3}[x]$.  They constructed a few classes of permutation trinomials   of that form
	by means of the iterative method, the multivariate method and  the resultant elimination. Motivated by \cite{Pang},   Zheng et al. \cite{Zheng4} studied two classes of permutation trinomials  over ${\mathbb F}_{q^3}$ of the form $f(x)=ax +L(x^s)$, where
	$L(x)$ is a linearized polynomial, by transforming the problem into dealing with some equations over
	${\mathbb F}_{q^3}$. 
	Very recently, Qu et al. \cite{Qu} studied  a bijection from
	the multiplicative subgroup  $U_{q^2+q+1}$  of 	${\mathbb F}_{q^3}$ with order
	$q^2+q+1$ to the projective plane ${\rm  PG}(2, q)$. In that paper, some explicit permutation polynomials of the form $x^rh(x^{q-1})$
	over ${\mathbb F}_{q^3}$  were constructed via bijections of ${\rm  PG}(2, q)$.

    Permutation pentanomials over $\mathbb{F}_{q^3}$ 
	have not been well investigated so far, and  there are very few related conclutions.
	 Rencently,   Zhang and Zheng \cite{Zhang3} investigated several classes of permutation pentanomials  of the form
	 \begin{equation}\label{0}
	 \epsilon_0x^{d_0}+L(\epsilon_{1}x^{d_1}+\epsilon_{2}x^{d_2})
	 \end{equation}
	 over ${\mathbb F}_{q^3}~(d_0=1,2,4)$
	from some certain  linearized polynomial $L(x)$ by using different methods. Numerical
	results suggest that there exist large classes of permutation pentanomials of   form (\ref{0})
	over $\mathbb{F}_{q^3}$.  
	 However,  some   permutation pentanomials  of form (\ref{0})  over $\mathbb{F}_{q^3}$  have not been proven yet.
	 
	In 2023,   Li and Nikolay \cite{LiKK} presented two infinite families of {\rm APN} functions in triviariate form over finite fields of the form $\mathbb{F}_{2^{3m}}$.
	They also proved the permutation property of the functions from both families when $m$ is odd. Inspired by their methods     in that paper,  we further characterize the  
	permutation property of the  pentanomials with    form (\ref{0})
	over $\mathbb{F}_{q^3}$.

 The remainder of this paper is organized as follows. In Section \ref{section 2}, we recall some basic notions and related results.
 In Section \ref{section 4}, we
 investigate several classes of permutation  pentanomials over $\mathbb{F}_{q^3}$ of the form form (\ref{0}) with $L(x)=x+x^q$.  In
 Section \ref{section 7}, we   discuss the quasi-multiplicative equivalence between
 permutation pentanomials proposed in this paper and the known ones. In Section \ref{section 8}, we  conclude this paper.

\section{Preliminaries}\label{section 2}

 Throughout this paper, let  $m,l$  be two positive integers and $q=2^m$, we use ${\rm Tr}_{q^l/q}(\cdot)$ to denote the \emph{trace function} from
 ${\mathbb F}_{q^l}$ to ${\mathbb F}_{q}$, that is
 $${\rm Tr}_{q^l/q}(x)=x+x^{q}+x^{q^2}+\cdots +x^{q^{l-1}}.$$
 If $q=2$, ${\rm Tr}_{q^l/q}(\cdot)$ is called the absolute trace function, and it is simply deoted by ${\rm Tr}_{q^l}(\cdot)$.

  Let $k$ be a positive integer, if ${\rm gcd}(k,q-1)=1$, it is clear that $x^k$ is a permutation of ${\mathbb F}_{q}$. For $\alpha\in {\mathbb F}_{q}$,   denote by $\alpha^{\frac{1}{k}}$ the unique element $A$ in ${\mathbb F}_{q}$ such that $A^k=\alpha$.

  We need the following definition.

 \begin{defn}(\cite{LN})
 	Let $f(x) = a_0x^n + a_1x^{n-1}+\cdots+a_n\in{\mathbb F}_{q}[x]$ and $g(x) = b_0x^m + b_1x^{m-1}+\cdots+b_m\in{\mathbb F}_{q}[x]$ be two polynomials
 	of  degree $n>0$ and $m>0$, respectively.
 	Then the resultant 	of  the two polynomials with respect to $x$ is defined by the determinant

 	$$
 	R(f,g,x)=
 	\begin{vmatrix}
 	a_0& a_1 & \cdots&  a_n & 0&\cdots&\cdots&0\\
 	0 & a_0 & a_1 & \cdots & a_n&0&\cdots&0  \\
 	\vdots& \vdots & \vdots & \vdots &\vdots&\vdots&\vdots&\vdots\\
 	0 & \cdots & 0 & a_0 & a_1&a_2&\cdots&a_n  \\
 	b_0& b_1 & \cdots&  \cdots & b_m&0&\cdots&0\\
 	0 & b_0 & b_1 & \cdots & \cdots&b_m&\cdots&0  \\
 	\vdots& \vdots & \vdots & \vdots &\vdots&\vdots&\vdots&\vdots\\
 	0 & \cdots & 0 & b_0 & b_1&\cdots&\cdots&b_m  \\
 	\end{vmatrix}
 	$$
 \end{defn}
 of order $m+n$.

If $f(x)=a_0(x-\alpha_1)(x-\alpha_2)\cdots(x-\alpha_n)$, where $a_0\neq 0$, in the splitting field of $f$ over ${\mathbb F}_{q}$, then $R(f,g,x)$ is also given by the formula
$$R(f,g,x)=a_0^m\prod_{i=1}^ng(\alpha_i).$$
It is well known that $R(f,g,x)=0$ if and only if $f$ and $g$ have a common root in
${\mathbb F}_{q}$, which is the same as saying that $f$ and $g$ have a common divisor in ${\mathbb F}_{q}[x]$ of positive degree (\cite{LN}).

 The following lemma  will be used in our proof.   

\begin{lem}\label{two}\cite{Li2}
	Let $a,b \in \mathbb{F}_{q}$, where $q=2^m$ and $a\neq 0.$  Then $x^2+ax+b=0$ has exactly two solutions in $\mathbb{F}_{q}$ if and only if ${\rm Tr}_{q}\Big(\frac{b}{a^2}\Big)=0$; otherwise, it has no solutions in $ \mathbb{F}_{q} $.
\end{lem}

	Let $ q=2^m$ and    $f(x)=x^3+ax+b\in \mathbb{F}_{q}[x]$. If $f(x)$ factors over $\mathbb{F}_{q}$ as a
product of three linear factors, we write $f=(1,1,1)$; if $f(x)$ factors over $\mathbb{F}_{q}$ as a
product of a linear factor and an irreducible quadratic factor, we write
$f = (1, 2)$; and  if $f(x)$ is itself irreducible over $\mathbb{F}_{q}$, we write $f = (3)$.
Let $ a, b\in \mathbb{F}_{q} $ with $ b\neq 0 $. If  $ {\rm Tr}_{q}\Big(\frac{a^3}{b^2}\Big)={\rm Tr}_{q}(1) $, according to Lemma \ref{two}, we have that  $ x^2+bx+a^3=0 $ has exactly two solutions in $ \mathbb{F}_{q} $ if $ m $ is even (in $ \mathbb{F}_{q^{2}}\backslash \mathbb{F}_{q}  $ if $ m $ is odd, respectively). Denote these two solutions by $ y_{1} $ and $ y_{2} $ when $ {\rm Tr}_{q}\Big(\frac{a^3}{b^2}\Big)={\rm Tr}_{q}(1) $.
The following lemma   describes the factorization of a cubic polynomial over  $\mathbb{F}_{q}$.

\begin{lem} \label{three}\cite{KS}
	The factorization of  $f(x)=x^3+ax+b~(b\neq 0)$  over  $\mathbb{F}_{q}$  is characterized as follows:

	(1) $f=(1,1,1)$ if and only if ${\rm Tr}_{q}\Big(\frac{a^3}{b^2}\Big)={\rm Tr}_{q}(1),$ $y_{1}$ and $y_{2}$ cubes in $\mathbb{F}_{q}$ ($m$ even), $\mathbb{F}_{q^2}$ ($m$ odd);
	
	(2) $f=(1,2)$  if and only if ${\rm Tr}_{q}\Big(\frac{a^3}{b^2}\Big)\neq {\rm Tr}_{q}(1)$;
	
	(3) $f=(3)$   if and only if  ${\rm Tr}_{q}\Big(\frac{a^3}{b^2}\Big)={\rm Tr}_{q}(1),$ $y_{1}$ and $y_{2}$  not cubes in $\mathbb{F}_{q}$ ($m$ even), $\mathbb{F}_{q^{2}}$ ($m$ odd).
\end{lem}

 Leonard and  Williams  characterized the factorization of the quartic polynomial
$f(x)=x^4+a_2x^2+a_1x+a_0$ over $\mathbb{F}_{q}$, where $q=2^m$. Similar to
 the statement above Lemma \ref{three}, the
 shorthand will be used in the following lemma. For example, if $f(x)$
factors as a product of two linear factors times an irreducible
quadratic, we write $f=(1,1,2)$.
We use $r_1,r_2,r_3$ below to indicate roots of the cubic  equation  $g(y)=y^3+a_2y+a_1$
when they exist in $\mathbb{F}_{q}$, and set $w_i=a_0r_{i}^2/a_{1}^{2}$ in this case.

\begin{lem} \label{four}\cite{Le}
	The factorization of  $f(x)=x^4+a_2x^2+a_1x+a_0$  over  $\mathbb{F}_{q}$  is characterized as follows:

	(1) $f=(1,1,1,1)$ if and only if $g=(1,1,1)$ and ${\rm{Tr}}_{q}(w_1)={\rm{Tr}}_{q}(w_2)={\rm{Tr}}_{q}(w_3)=0$;
	
	(2)  $f=(2,2)$ if and only if $g=(1,1,1)$  and ${\rm{Tr}}_{q}(w_1)=0,{\rm{Tr}}_{q}(w_2)={\rm{Tr}}_{q}(w_3)=1$;
	
	(3) $f=(1,3)$   if and only if   $g=(3)$;
	
	(4) $f=(1,1,2)$   if and only if   $g=(1,2)$ and ${\rm{Tr}}_{q}(w_1)=0$;
	
	(5) $f=(4)$   if and only if   $g=(1,2)$ and ${\rm{Tr}}_{q}(w_1)=1$.
	
\end{lem}

\section{Permutation  pentanomials  over $\mathbb{F}_{q^3}$}\label{section 4}

In this section, we shall present  several classes of permutation  pentanomials  over  $\mathbb{F}_{q^3}$ of the form
$\epsilon_{0}x^{d_0}+L(\epsilon_{1}x^{d_1}+\epsilon_{2}x^{d_2})$ for $d_0=1, 2, 4$, here $L(x)=x+x^q$.

\begin{thm}\label{1}
	Let    $m$ be an  integer with $  q=2^m$. Then
	\begin{equation*}
	f(x)=  x + x^{2q^2+2}+x^{2q+2}+x^{2}+x^{2q}
	\end{equation*}
	is a permutation polynomial over $\mathbb{F}_{q^3}$. 	
\end{thm}
\begin{proof}
In the following, we shall show that, for any  $ a\in\mathbb{F}^*_{q^3}$,  the equations $f(x+a)=f(x)$ 	 has no solution in $\mathbb{F}_{q^3}$. Assume, on the
 contrary, that there exist $x\in\mathbb{F}_{q^3}$ and  $ a\in\mathbb{F}^*_{q^3}$ such that $f(x+a)=f(x)$. Let $y=x^q,z=x^{q^2},b=a^q,c=a^{q^2}$. Then we have  the following equations  of system
 \begin{align*}\label{2.1}
 \begin{split}
 \left \{
 \begin{array}{ll}
 (x+a)+(x+a)^2(z+c)^2+(x+a)^2(y+b)^2+(x+a)^2+(y+b)^2=x+x^2z^2+x^2y^2+x^2+y^2,                   \\
 (y+b)+(y+b)^2(x+a)^2+(y+b)^2(z+c)^2+(y+b)^2+(z+c)^2=y+y^2x^2+y^2z^2+y^2+z^2,                  \\
  (z+c)+(z+c)^2(y+b)^2+(z+c)^2(x+a)^2+(z+c)^2+(x+a)^2=z+z^2y^2+z^2x^2+z^2+x^2,                  \\
 \end{array}
 \right.
 \end{split}
 \end{align*}
which can be reduced to  
	\begin{align*}
	\begin{split}
	\left \{
	\begin{array}{ll}
	(b^2+c^2)x^2+a^2y^2+a^2z^2+a+a^2c^2+a^2b^2+a^2+b^2=0,                   \\
	(a^2+c^2)y^2+b^2z^2+b^2x^2+b+b^2a^2+b^2c^2+b^2+c^2=0,                  \\
	(a^2+b^2)z^2+c^2x^2+c^2y^2+c+c^2b^2+c^2a^2+c^2+a^2=0.                  \\
	\end{array}
	\right.
	\end{split}
	\end{align*}	
Adding the  three equations above, we get	$a+b+c=0.$ Plugging $c=a+b$ into the first two equations of the above system, we have 
$$x^2+y^2+z^2=\frac{a+a^4+a^2+b^2}{a^2}=\frac{b+b^4+a^2}{b^2}.$$
Therefore, we get the following equations, 
\begin{equation*}
\begin{split}
&ab^2+a^4b^2+a^2b^2+b^4+a^2b+a^2b^4+a^4\\
&= (ab+a+b)(ab+a+b^2)(a^2+ab+b)=0.
\end{split}
\end{equation*}	
	
	The following proof is divided into two cases.
	
   \textbf{Case 1.} $ab+a+b=0$.	

	In this case, we have $a\neq 1$ and $b=\frac{a}{a+1}$. Then $c=\frac{b}{b+1}=a$, which contradicts the facts  $a+b+c=0$ and $a\neq 0$.
	
	\textbf{Case 2.} $ab+a+b^2=0$.	
	
	In this case, we have $b\neq 1$, $a=\frac{b^2}{b+1}$. Then    $b=\frac{c^2}{c+1}$ and $a=\frac{c^4}{c^3+1}$. Therefore, we have  
	$$a+b+c=\frac{c^4}{c^3+1}+\frac{c^2}{c+1}+c=\frac{c(c+1)^3}{c^3+1},$$
	which implies that $c=0$ or $1$,  a contradiction.
	
	Similarly, we have that $a^2+ab+b\neq 0.$   The proof is complete.
\end{proof}

	The proof of Theorem  \ref{1} is inspired by the idea of \cite{LiKK}. 
	
	\begin{thm}
	Let    $m$ be an  integer with $  q=2^m$. Then
	\begin{equation*}
	f(x)=  x + x^{4q^2+1}+x^{q+4}+x^{q^2+4}+x^{4q+1}
	\end{equation*}
	is a permutation polynomial over $\mathbb{F}_{q^3}$ if and only if $m$ is odd. 	
	\end{thm}
	\begin{proof}
    It suffices to show that $f(x+a)=f(x)$ has no solution in $\mathbb{F}_{q^3}$  
   for any  $ a\in\mathbb{F}^*_{q^3}$.	 
    Let $y=x^q,z=x^{q^2},b=a^q,c=a^{q^2}$. Then we have the following equations of system, 
    \begin{align*}
    \begin{split}
    \left \{
    \begin{array}{ll}
    (x+a)+(x+a)(z+c)^4+(x+a)^4(y+b)+(x+a)^4(z+c)+(x+a)(y+b)^4=x+xz^4+x^4y+x^4z+xy^4,                   \\
    (y+b)+(y+b)(x+a)^4+(y+b)^4(z+c)+(y+b)^4(x+a)+(y+b)(z+c)^4=y+yx^4+y^4z+y^4x+yz^4,                  \\
    (z+c)+(z+c)(y+b)^4+(z+c)^4(x+a)+(z+c)^4(y+b)+(z+c)(x+a)^4=z+zy^4+z^4x+z^4y+zx^4,                  \\
    \end{array}
    \right.
    \end{split}
    \end{align*}
    which can be reduced to 
   	\begin{align*}
   \begin{split}
   \left \{
   \begin{array}{ll}
   (b+c)x^4+(b^4+c^4)x+ay^4+a^4y+az^4+a^4z+a^4b+a^4c+ab^4+ac^4+a=0,                   \\
   (c+a)y^4+(c^4+a^4)y+bz^4+b^4z+bx^4+b^4x+b^4c+b^4a+bc^4+ba^4+b=0,                  \\
   (a+b)z^4+(a^4+b^4)z+cx^4+c^4x+cy^4+c^4y+c^4a+c^4b+ca^4+cb^4+c=0.                  \\
   \end{array}
   \right.
   \end{split}
   \end{align*}	
Adding the  three equations above, we get	$a+b+c=0.$ Plugging $c=a+b$ into the first two equations of the above system, we have 
	\begin{align*}
\begin{split}
\left \{
\begin{array}{ll}
(x+y+z)^4+a^3(x+y+z)=1,                   \\
(x+y+z)^4+b^3(x+y+z)=1.                  \\
\end{array}
\right.
\end{split}
\end{align*}
Therefore,     $a^3=b^3=c^3$. Noting that $m$ is odd, we get $a=b=c,$ which contradicts the facts $a+b+c=0$ and $a\neq 0$.

Conversely, if $m$ is even, then $g(x)=x^3$ is not a permutation polynomial over $\mathbb{F}_{q}$. There exists an element $\alpha\in\mathbb{F}^*_{q}$ such that $x^3+\alpha=0$ has  no solutions in $\mathbb{F}_{q}$. By Lemma \ref{four}, we have that 
$x^4+\alpha x+1=0$ has only one solution in $\mathbb{F}_{q}$, denoted as $\gamma_0$.
Then we have $\alpha=\frac{\gamma_0^4+1}{\gamma_0}$.

Assume that $\omega$ is a primitive 	element of  ${\mathbb F}_{q^3}$, and $\alpha=\omega^{(q^2+q+1)t}$    for some positive integer $t$. Let $\beta=\omega^{\frac{(q^2+q+1)}{3}t}$ and $\eta=\omega^{\frac{q^3-1}{3}}$.
Then we have 
$x^3+\alpha=(x+\beta)(x+\eta\beta)(x+\eta^2\beta)$. By  Lemma \ref{four}, we have that 
$x^4+\alpha x+1=0$ has four solutions in $\mathbb{F}_{q^3}$. Assume that $\gamma_1\in\mathbb{F}_{q^3}\backslash\mathbb{F}_{q}$  is a solution  of $x^4+\alpha x+1=0$, which implies that $\gamma_1^4+\alpha\gamma_1+1=\gamma_1^4+\frac{\gamma_0^4+1}{\gamma_0}\gamma_1+1=0$. 
It is easy to check that $\gamma_1+\gamma_1^q+\gamma_1^{q^2}\in\mathbb{F}_{q}$ is a solution of $x^4+\alpha x+1=0$. Therefore, $\gamma_1+\gamma_1^q+\gamma_1^{q^2}=\gamma_0$.
Then we have the following equations,
\begin{equation*}
\begin{split}
f(\gamma_1)&=\gamma_1+ \gamma_1^{4q^2+1}+\gamma_1^{q+4}+\gamma_1^{q^2+4}+\gamma_1^{4q+1}\\
&=\gamma_1+\gamma_1(\gamma_1^q+\gamma_1)^{4q}+\gamma_1^4(\gamma_1^q+\gamma_1)^q\\
&=\gamma_1+\gamma_1(\gamma_1^{q^2}+\gamma_0)^{4q}+\gamma_1^4(\gamma_1^{q^2}+\gamma_0)^q\\
&=\gamma_1+\gamma_1\gamma_0^4+\gamma_1^4\gamma_0\\
&=\gamma_0(\gamma_1^4+\frac{\gamma_0^4+1}{\gamma_0}\gamma_1)\\
&=\gamma_0=f(\gamma_0),
\end{split}
\end{equation*}
 a contradiction.
 We complete the proof.
\end{proof}	

\begin{thm}
 Let    $m$ be an  integer with $  q=2^m$. Then
		\begin{equation*}
	f(x)=  x + x^{q^2+2}+x^{2q+1}+x^{3}+x^{3q}
	\end{equation*}
is a permutation polynomial over $\mathbb{F}_{q^3}$ if and only if $m$ is odd. 	
\end{thm}

\begin{proof}
	For each fixed  $a\in \mathbb{F}_{q^3}$, it suffices to prove that  the following equation
	\begin{equation*}
     x + x^{q^2+2}+x^{2q+1}+x^{3}+x^{3q}=a
	\end{equation*} 	
	has at most one solution in  $\mathbb{F}_{q^3}$.

	 Let $y=x^q,z=y^q,b=a^q,c=b^q$ and $A=a+b+c.$  Then we have the following equations of system, 
	\begin{align}\label{1.1}
	\begin{split}
	\left \{
	\begin{array}{ll}
	x+x^2z+xy^2+x^3+y^3=a,                   \\
	y+y^2x+yz^2+y^3+z^3=b,                 \\
	z+z^2y+zx^2+z^3+x^3=c.                  \\
	\end{array}
	\right.
	\end{split}
	\end{align}
	Adding the above three equations, we get $x+y+z=a+b+c=A$. Plugging  $z=x+y+A$ into the first two equations of System  (\ref{1.1}), we can obtain
	\begin{equation}\label{1.2}
	f_1: x+x^2y+Ax^2+xy^2+y^3+a=0,
	\end{equation} 		
	\begin{equation}\label{1.3}
	f_2: x^3+Ax^2+A^2x+y^3+Ay^2+y+A^3+b=0.
	\end{equation} 		
	By   Magma, the resultant of $f_1$ and $f_2$
	with respect to $y$ is
	\begin{equation*}
	\begin{split}
	&R(f_1,f_2,y)=(a^2 +ab + aA + b^2 + bA + A^2) x^3+\\
	& (a^2A + abA +aA^2 +b^2A +
	bA^4 +bA^2 +b +A^5 +A^3 +A)x^2+\\
	 & (a^2 + ab + aA + b^2A^2 + bA + A^8 + 
	A^4 + 1)x + \\
	&a^3 + a^2b + a^2A + ab^2 + abA + aA^6 + aA^4 + a + b^3 + b^2A^3 + bA^6 + A^9.
	\end{split}
	\end{equation*}
	Since $f_1$ and $f_2$ have a common root,
	we have
	$R(f_1,f_2,y)=0.$
	Let 
	\begin{equation*}
	\begin{split}
		&R=a^2 +ab + aA + b^2 + bA + A^2=a^2+b^2+c^2+ab+ac+bc,\\
		& S=a^2A + abA +aA^2 +b^2A +
		bA^4 +bA^2 +b +A^5 +A^3 +A, \\
		& T=a^2 + ab + aA + b^2A^2 + bA + A^8 + 
		A^4 + 1, \\
		&U=a^3 + a^2b + a^2A + ab^2 + abA + aA^6 + aA^4 + a + b^3 + b^2A^3 + bA^6 + A^9.
	\end{split}
\end{equation*}
	Then 	\begin{equation}\label{1.21}
	R(f_1,f_2,y)=Rx^3+Sx^2+Tx+U=0.
	\end{equation} 	
	
	 We claim that $R=0$ if and only if $a\in \mathbb{F}_{q}$. The sufficiency is obvious.
	 Assume that 
 	 $R=a^2 +ab + aA + b^2 + bA + A^2=a^2+b^2+c^2+ab+ac+bc=0$, then $(a+b)^2+(a+c)^2+(a+b)(a+c)=0$. If $a\neq b$,  we get
		\begin{equation*} 
  \Big(\frac{a+c}{a+b}\Big)^2+\frac{a+c}{a+b}+1=0.
	\end{equation*} 	
	Note that $m$ is odd.   Applying ${\rm Tr}_{q^3}(\cdot)$ on both sides of the above equation, we get  
	${\rm Tr}_{q^3}(1)=3m=1\neq 0$, which is a contradiction. Then we have $a=b$, which implies   $a\in \mathbb{F}_{q}$. 
	
	The following proof is divided into two cases.
	
	\textbf{Case 1.} $R=0$.
	
	In this case, we have $a=b=A$.
	Then from Eq. (\ref{1.21}), we have $R(f_1,f_2,y)=(a^8+1)x+a^9+a=0$. If $a\neq 1$, we get $x=a.$
     If $a=1$,	
	 then Eqs. (\ref{1.2}) and (\ref{1.3}) becomes
	\begin{equation*}
	f_1: x+x^2y+x^2+xy^2+y^3+1=0,
	\end{equation*} 		
	\begin{equation*}
	f_2: x^3+x^2+x+y^3+y^2+y=0.
	\end{equation*} 		
	Let $y=\tau x$. Then the above two equations become 
	\begin{equation}\label{1.7}
	f_1: \tau(\tau^2+\tau+1)x^3+x^2+x+1=0,
	\end{equation} 		
	\begin{equation}\label{1.8}
	f_2: x((\tau^3+1)x^2+(\tau^2+1)x+\tau+1)=0.
	\end{equation} 	
	It is clear that $x\neq 0$.
	If $\tau=1$, then we have $x=a=1$ from Eq.  (\ref{1.7}). If $\tau\neq 1$,
	we have $\tau^3+1 \neq 0$, since $m$ is odd.
	From Eq. (\ref{1.8}), we have 
	\begin{equation}\label{1.9}
	 x^2+\frac{\tau+1}{\tau^2+\tau+1}x+\frac{1}{\tau^2+\tau+1}=0.
	\end{equation} 	
	Since $${\rm Tr}_{q^3}\Big(\frac{1}{\tau^2+\tau+1}\Big(\frac{\tau^2+\tau+1}{\tau+1}\Big)^2\Big)={\rm Tr}_{q^3}\Big(\frac{\tau^2+\tau+1}{\tau^2+1}\Big)={\rm Tr}_{q^3}\Big(1+\frac{1}{\tau+1}+\frac{1}{\tau^2+1}\Big)={\rm Tr}_{q^3}(1)=3m=1,$$
	we have that Eq.  (\ref{1.9}) has no solutions in $\mathbb{F}_{q^3}$ by Lemma \ref{two}.

	 	\textbf{Case 2.} $R\neq 0$.
	
	From Eq. (\ref{1.21}),  we have $x^3+R^{-1}Sx^2+R^{-1}Tx+R^{-1}U=0$. Let $x=X+R^{-1}S$.
	Then the above equation becomes
	\begin{equation}\label{1.10}
	X^3+R^{-2}(RT+S^2)X+R^{-2}(RU+ST)=0.
	\end{equation} 

\textbf{Subcase 2-1.} $RU+ST= 0$. 

	 In this case, we claim that  $RT+S^2=0$.
	Let 
	$h_1=a^6 + a^4 b^2 + a^4 c^2 + a^4 + a^2 b^4 + a^2 c^4 + a b + a c + b^6 + b^4 c^2 + b^4 + b^2 c^4 + b c + c^6 + c^4 +
	1,$ and
	$h_2=a^7 + a^6 c + a^5 b^2 + a^5 c^2 + a^5 + a^4 b^2 c + a^4 c^3 + a^4 c + a^3 b^4 + a^3 c^4 + a^2 b^4 c + a^2 b +
	a^2 c^5 + a^2 c + a b^6 + a b^4 c^2 + a b^4 + a b^2 c^4 + a c^6 + a c^4 + a c^2 + a + b^6 c + b^4 c^3 + b^4 c + 
	b^3 + b^2 c^5 + b^2 c + c^7 + c^5 + c^3 + c.$
	It is easy to verify that $RU+ST=h_1h_2$ and $RT+S^2=(b+c)(a+b)(a+b+c+1)^2h_1$.
	If $h_1=0$, then we  have $RT+S^2=0$. If $h_2=0$, noting that $c=b^{q},b=a^q$, we have
		\begin{equation*} 
	h_2+h_2^q+h_2^{q^2}=a^2b+a^2c+ab^2+ac^2+b^2c+bc^2=(a+b)(a+c)(b+c)=0,
	\end{equation*}
	which implies that $a=b=c$. Therefore,  $RT+S^2=(b+c)(a+b)(a+b+c+1)^2h_1=0$.
	In this case, $X=0$ from Eq. (\ref{1.10}). Furthermore, we have  $x=X+R^{-1}S=R^{-1}S$.
	
\textbf{Subcase 2-2.} $RU+ST\neq 0$. 	
	
	Considering Eq. (\ref{1.10}), in the following, we shall show that 
	$${\rm Tr}_{q^3}\Big(   \frac{  R^{-6}(RT+S^2)^3  }{  R^{-4}(RU+ST)^2   }      \Big)={\rm Tr}_{q^3}\Big(   \frac{  (RT+S^2)^3  }{  R^2(RU+ST)^2   }      \Big)={\rm Tr}_{q^3}\Big(   \frac{  (b+c)^3(a+b)^3(a+b+c+1)^6h_1  }{  R^2h_2^2   }      \Big)=0.$$
	By Lemma (\ref{two}), it is  equivalent to show that the equation
	$$Y^2+Y+\frac{  (b+c)^3(a+b)^3(a+b+c+1)^6h_1  }{  R^2h_2^2   } =0$$ has two solutions in
	 $ \mathbb{F}_{q^3}$.
	Let $Y=\frac{Z}{Rh_2}$. Then the above equation becomes 
	$$Z^2+Rh_2Z+(b+c)^3(a+b)^3(a+b+c+1)^6h_1=0.$$ 
	Let $V_1=(b+c)^3(a+b)^3(a+b+c+1)^6h_1, V_2=Rh_2$.
	With the help of Magma, we have 
	$V_1 =c^{15} a^{3} + c^{15} a^{2} b + c^{15} a b^{2} + c^{15} b^{3} + c^{14} a^{3} b + c^{14} a^{2} b^{2} + c^{14} a b^{3} + c^{14} b^{4} + c^{13} a^{3} b^{2} +
	c^{13} a^{2} b^{3} + c^{13} a b^{4} + c^{13} b^{5} + c^{12} a^{3} b^{3} + c^{12} a^{2} b^{4} + c^{12} a b^{5} + c^{12} b^{6} + c^{11} a^{7} + c^{11} a^{6} b +
	c^{11} a^{5} b^{2} + c^{11} a^{4} b^{3} + c^{11} a^{3} b^{4} + c^{11} a^{2} b^{5} + c^{11} a b^{6} + c^{11} b^{7} + c^{10} a^{7} b + c^{10} a^{6} b^{2} +
	c^{10} a^{5} b^{3} + c^{10} a^{4} b^{4} + c^{10} a^{4} + c^{10} a^{3} b^{5} + c^{10} a^{2} b^{6} + c^{10} a b^{7} + c^{10} b^{8} + c^{10} b^{4} +
	c^{9} a^{7} b^{2} + c^{9} a^{6} b^{3} + c^{9} a^{5} b^{4} + c^{9} a^{4} b^{5} + c^{9} a^{3} b^{6} + c^{9} a^{3} b^{2} + c^{9} a^{3} + c^{9} a^{2} b^{7} +
	c^{9} a^{2} b^{3} + c^{9} a^{2} b + c^{9} a b^{8} + c^{9} a b^{4} + c^{9} a b^{2} + c^{9} b^{9} + c^{9} b^{5} + c^{9} b^{3} + c^{8} a^{7} b^{3} +
	c^{8} a^{6} b^{4} + c^{8} a^{6} + c^{8} a^{5} b^{5} + c^{8} a^{4} b^{6} + c^{8} a^{4} b^{2} + c^{8} a^{4} + c^{8} a^{3} b^{7} + c^{8} a^{3} b^{3} + c^{8} a^{3} b + 
	c^{8} a^{2} b^{8} + c^{8} a^{2} b^{2} + c^{8} a b^{9} + c^{8} a b^{5} + c^{8} a b^{3} + c^{8} b^{10} + c^{7} a^{11} + c^{7} a^{10} b + c^{7} a^{9} b^{2} +
	c^{7} a^{8} b^{3} + c^{7} a^{5} b^{2} + c^{7} a^{5} + c^{7} a^{4} b^{3} + c^{7} a^{4} b + c^{7} a^{3} b^{8} + c^{7} a^{3} b^{4} + c^{7} a^{2} b^{9} +
	c^{7} a^{2} b^{5} + c^{7} a b^{10} + c^{7} a b^{4} + c^{7} b^{11} + c^{7} b^{5} + c^{6} a^{11} b + c^{6} a^{10} b^{2} + c^{6} a^{9} b^{3} + c^{6} a^{8} b^{4} + 
	c^{6} a^{8} + c^{6} a^{5} b^{3} + c^{6} a^{5} b + c^{6} a^{4} b^{2} + c^{6} a^{4} + c^{6} a^{3} b^{9} + c^{6} a^{3} b^{5} + c^{6} a^{2} b^{10} + c^{6} a^{2} b^{6} +
	c^{6} a b^{11} + c^{6} a b^{5} + c^{6} b^{12} + c^{6} b^{6} + c^{6} b^{4} + c^{5} a^{11} b^{2} + c^{5} a^{10} b^{3} + c^{5} a^{9} b^{4} + c^{5} a^{8} b^{5} +
	c^{5} a^{7} b^{2} + c^{5} a^{7} + c^{5} a^{6} b^{3} + c^{5} a^{6} b + c^{5} a^{3} b^{10} + c^{5} a^{3} b^{6} + c^{5} a^{3} b^{2} + c^{5} a^{3} + c^{5} a^{2} b^{11} 
	+ c^{5} a^{2} b^{7} + c^{5} a^{2} b^{3} + c^{5} a^{2} b + c^{5} a b^{12} + c^{5} a b^{6} + c^{5} a b^{4} + c^{5} a b^{2} + c^{5} b^{13} + c^{5} b^{7} +
	c^{5} b^{5} + c^{5} b^{3} + c^{4} a^{11} b^{3} + c^{4} a^{10} b^{4} + c^{4} a^{10} + c^{4} a^{9} b^{5} + c^{4} a^{8} b^{6} + c^{4} a^{8} b^{2} + c^{4} a^{8} +
	c^{4} a^{7} b^{3} + c^{4} a^{7} b + c^{4} a^{6} b^{2} + c^{4} a^{6} + c^{4} a^{4} b^{6} + c^{4} a^{4} b^{4} + c^{4} a^{4} b^{2} + c^{4} a^{4} + c^{4} a^{3} b^{11} +
	c^{4} a^{3} b^{7} + c^{4} a^{3} b^{3} + c^{4} a^{3} b + c^{4} a^{2} b^{12} + c^{4} a^{2} b^{8} + c^{4} a^{2} b^{2} + c^{4} a b^{13} + c^{4} a b^{7} +
	c^{4} a b^{5} + c^{4} a b^{3} + c^{4} b^{14} + c^{4} b^{8} + c^{3} a^{15} + c^{3} a^{14} b + c^{3} a^{13} b^{2} + c^{3} a^{12} b^{3} + c^{3} a^{11} b^{4} +
	c^{3} a^{10} b^{5} + c^{3} a^{9} b^{6} + c^{3} a^{9} b^{2} + c^{3} a^{9} + c^{3} a^{8} b^{7} + c^{3} a^{8} b^{3} + c^{3} a^{8} b + c^{3} a^{7} b^{8} +
	c^{3} a^{7} b^{4} + c^{3} a^{6} b^{9} + c^{3} a^{6} b^{5} + c^{3} a^{5} b^{10} + c^{3} a^{5} b^{6} + c^{3} a^{5} b^{2} + c^{3} a^{5} + c^{3} a^{4} b^{11} +
	c^{3} a^{4} b^{7} + c^{3} a^{4} b^{3} + c^{3} a^{4} b + c^{3} a^{3} b^{12} + c^{3} a^{3} b^{8} + c^{3} a^{3} b^{4} + c^{3} a^{3} + c^{3} a^{2} b^{13} +
	c^{3} a^{2} b^{9} + c^{3} a^{2} b^{5} + c^{3} a^{2} b + c^{3} a b^{14} + c^{3} a b^{8} + c^{3} a b^{4} + c^{3} a b^{2} + c^{3} b^{15} + c^{3} b^{9} +
	c^{3} b^{5} + c^{3} b^{3} + c^{2} a^{15} b + c^{2} a^{14} b^{2} + c^{2} a^{13} b^{3} + c^{2} a^{12} b^{4} + c^{2} a^{11} b^{5} + c^{2} a^{10} b^{6} +
	c^{2} a^{9} b^{7} + c^{2} a^{9} b^{3} + c^{2} a^{9} b + c^{2} a^{8} b^{8} + c^{2} a^{8} b^{2} + c^{2} a^{7} b^{9} + c^{2} a^{7} b^{5} + c^{2} a^{6} b^{10} +
	c^{2} a^{6} b^{6} + c^{2} a^{5} b^{11} + c^{2} a^{5} b^{7} + c^{2} a^{5} b^{3} + c^{2} a^{5} b + c^{2} a^{4} b^{12} + c^{2} a^{4} b^{8} + c^{2} a^{4} b^{2} +
	c^{2} a^{3} b^{13} + c^{2} a^{3} b^{9} + c^{2} a^{3} b^{5} + c^{2} a^{3} b + c^{2} a^{2} b^{14} + c^{2} a^{2} b^{10} + c^{2} a^{2} b^{6} + c^{2} a^{2} b^{2} +
	c^{2} a b^{15} + c^{2} a b^{9} + c^{2} a b^{5} + c^{2} a b^{3} + c^{2} b^{16} + c^{2} b^{12} + c^{2} b^{10} + c^{2} b^{8} + c^{2} b^{6} + c^{2} b^{4} +
	c a^{15} b^{2} + c a^{14} b^{3} + c a^{13} b^{4} + c a^{12} b^{5} + c a^{11} b^{6} + c a^{10} b^{7} + c a^{9} b^{8} + c a^{9} b^{4} + c a^{9} b^{2} +
	c a^{8} b^{9} + c a^{8} b^{5} + c a^{8} b^{3} + c a^{7} b^{10} + c a^{7} b^{4} + c a^{6} b^{11} + c a^{6} b^{5} + c a^{5} b^{12} + c a^{5} b^{6} +
	c a^{5} b^{4} + c a^{5} b^{2} + c a^{4} b^{13} + c a^{4} b^{7} + c a^{4} b^{5} + c a^{4} b^{3} + c a^{3} b^{14} + c a^{3} b^{8} + c a^{3} b^{4} +
	c a^{3} b^{2} + c a^{2} b^{15} + c a^{2} b^{9} + c a^{2} b^{5} + c a^{2} b^{3} + c a b^{16} + c a b^{12} + c a b^{8} + c a b^{4} + c b^{17} +
	c b^{13} + c b^{9} + c b^{5} + a^{15} b^{3} + a^{14} b^{4} + a^{13} b^{5} + a^{12} b^{6} + a^{11} b^{7} + a^{10} b^{8} + a^{10} b^{4} + a^{9} b^{9} +
	a^{9} b^{5} + a^{9} b^{3} + a^{8} b^{10} + a^{7} b^{11} + a^{7} b^{5} + a^{6} b^{12} + a^{6} b^{6} + a^{6} b^{4} + a^{5} b^{13} + a^{5} b^{7} + a^{5} b^{5} +
	a^{5} b^{3} + a^{4} b^{14} + a^{4} b^{8} + a^{3} b^{15} + a^{3} b^{9} + a^{3} b^{5} + a^{3} b^{3} + a^{2} b^{16} + a^{2} b^{12} + a^{2} b^{10} + a^{2} b^{8} +
	a^{2} b^{6} + a^{2} b^{4} + a b^{17} + a b^{13} + a b^{9} + a b^{5} + b^{18} + b^{12} + b^{8} + b^{6}$,
	and
	
	$V_2=c^ {9} + c^ {8} b + c^ {7} a^ {2} + c^ {7} + c^ {6} a^ {3} + c^ {6} a b^ {2} + c^ {6} b^ {3} + c^ {6} b + c^ {5} a^ {4} + c^ {5} a^ {2} b^ {2} + c^ {5} b^ {2} + c^ {5} + c^ {4} a^ {5} + 
	c^ {4} a^ {3} + c^ {4} a^ {2} b^ {3} + c^ {4} a^ {2} b + c^ {4} a b^ {4} + c^ {4} a b^ {2} + c^ {4} b^ {5} + c^ {4} b + c^ {3} a^ {6} + c^ {3} a^ {4} + c^ {3} a^ {2} b^ {4} +
	c^ {3} a^ {2} + c^ {3} b^ {4} + c^ {3} + c^ {2} a^ {7} + c^ {2} a^ {5} b^ {2} + c^ {2} a^ {4} b^ {3} + c^ {2} a^ {4} b + c^ {2} a^ {3} b^ {4} + c^ {2} a^ {2} b + c^ {2} a b^ {6} +
	c^ {2} b^ {7} + c^ {2} b^ {5} + c^ {2} b + c a^ {6} b^ {2} + c a^ {4} b^ {4} + c a^ {4} b^ {2} + c a^ {4} + c a^ {2} b^ {6} + c a^ {2} b^ {2} + c b^ {8} + c b^ {6} +
	c b^ {2} + a^ {9} + a^ {8} b + a^ {7} + a^ {6} b^ {3} + a^ {6} b + a^ {5} b^ {2} + a^ {4} b^ {5} + a^ {4} b + a^ {3} b^ {4} + a^ {3} b^ {2} + a^ {3} + a^ {2} b^ {7} +
	a^ {2} b^ {5} + a^ {2} b + a b^ {8} + a b^ {6} + a b^ {4} + a b^ {2} + b^ {5}.$
	
	Consider $V_1$, $V_2$ and  $Z$ as polynomials with respect to $c$.
	Write $C(V_1,i), C(V_2,i),  C(Z,i)$ the   coefficients of the $i$-degree term of $V_1$,   $V_2$ and $Z$, respectively.
	Then $C(V_2,9) C(Z,6)=C(V_1,15)$, which implies $C(Z,6)=(a+b)^3$.
	Furthermore, $C(V_2,9) C(Z,5)+C(V_2,8) C(Z,6)=C(V_1,14)$, then we have $C(Z,5)=0$.
	Similarly, we can get all the coefficients of $Z$, and $Z=(a^3 + a^2 b + a b^2 + b^3) c^6+(a^5 + a^4 b + a^3 + a^2 b + a b^4 + a b^2 + b^5 + b^3) c^4+(a^7 + a^6 b + a^5 b^2 + a^4 b^3 + a^3 b^4 + a^2 b^5 + a b^6 + b^7) c^2+
	(a^4+b^4) c+(a^9+b a^8+a^7+b a^6+b^2 a^5+(b^3+b) a^4+(b^4+b^2+1) a^3+(b^5+b^3+b) a^2+(b^8+b^6+b^4+b^2) a+b^9+b^7+b^3)$.
	Thus $${\rm Tr}_{q^3}\Big(   \frac{  R^{-6}(RT+S^2)^3  }{  R^{-4}(RU+ST)^2   }      \Big)=0\neq{\rm Tr}_{q^3}(1).$$ By Lemma \ref{three}, we have that Eq. (\ref{1.10}) has only one solution in $\mathbb{F}_{q^3}$, so does Eq. (\ref{1.21}).
	
	Therefore, for any fixed $a\in \mathbb{F}_{q^3}$, $f(x)=a$ has at most one solution in ${\mathbb F}_{q^3}$.

	Conversely, if $m$ is even, then we  assume that $\omega$ is a primitive 	element of  ${\mathbb F}_{q^3}$. Let $\eta=\omega^{\frac{(q^2+q+1)(q-1)}{3}}$ and $\alpha=\omega^{\frac{q^2+q+1}{3}}$. It is easy to check that $\eta^2+\eta+1=0$ and   $\alpha^{q}=\eta \alpha, \alpha^{q^2}=\eta^2 \alpha$.
	Therefore, 
	\begin{equation*}
	\begin{split}
	f(\alpha+1)&=(\alpha+1) + (\alpha+1)^{q^2+2}+(\alpha+1)^{2q+1}+(\alpha+1)^{3}+(\alpha+1)^{3q}\\
	&=\alpha+1+\eta^2\alpha+\eta\alpha\\
     &=1=f(1).
	\end{split}
	\end{equation*}
	Since $\alpha\neq 0$, we get that $f(x)$ is not an injective function.
	This completes the proof.
\end{proof}

\begin{thm}
	Let    $m$ be an  integer with $q=2^m$. Then
	\begin{equation*}
	f(x)=  x + x^{2q^2+2q}+x^{2q^2+2}+x^{4}+x^{4q}
	\end{equation*}
	is a permutation polynomial over $\mathbb{F}_{q^3}$ if and only if $m\not\equiv 1~({\rm mod}\ 3)$. 	 	
\end{thm}
\begin{proof}
In the following, we shall show that, for any  $ a\in\mathbb{F}^*_{q^3}$,  the equations $f(x+a)=f(x)$ 	 has no solution in $\mathbb{F}_{q^3}$. Assume, on the
contrary, that there exist $x\in\mathbb{F}_{q^3}$ and  $ a\in\mathbb{F}^*_{q^3}$ such that $f(x+a)=f(x)$. Let $y=x^q,z=x^{q^2},b=a^q,c=a^{q^2}$. Then we have  the following equations  of system
\begin{align*}
\begin{split}
\left \{
\begin{array}{ll}
(x+a)+(y+b)^2(z+c)^2+(x+a)^2(z+c)^2+(x+a)^4+(y+b)^4=x+y^2z^2+x^2z^2+x^4+y^4,                   \\
(y+b)+(z+c)^2(x+a)^2+(y+b)^2(x+a)^2+(y+b)^2+(z+c)^2=y+z^2x^2+y^2x^2+y^4+z^4,                  \\
(z+c)+(x+a)^2(y+b)^2+(z+c)^2(y+b)^2+(z+c)^2+(x+a)^2=z+x^2y^2+z^2y^2+z^4+x^4,                  \\
\end{array}
\right.
\end{split}
\end{align*}
which can be reduced to  
\begin{align*}
\begin{split}
\left \{
\begin{array}{ll}
c^2x^2+c^2y^2+(a^2+b^2)z^2+a+a^2c^2+b^2c^2+a^4+b^4=0,                   \\
a^2y^2+a^2z^2+(b^2+c^2)x^2+b+b^2a^2+c^2a^2+b^4+c^4=0,                  \\
b^2z^2+b^2x^2+(c^2+a^2)y^2+c+c^2b^2+a^2b^2+c^4+a^4=0.                  \\
\end{array}
\right.
\end{split}
\end{align*}	
Adding the  three equations above, we get	$a+b+c=0.$ Plugging it into the three equations of the above system, we have 
	\begin{equation}\label{4.1}
x^2+y^2+z^2=\frac{b}{a^2}=\frac{c}{b^2}=\frac{a}{c^2}.
\end{equation}
Then we have $a^3+a^2b+b^3=0$. Therefore, we have 
$\frac{a}{b}$ is a solution of $x^3+x^2+1=0$. It is easy to check that $\frac{a}{b}\in\mathbb{F}_{2^3}$.

  We divide the following proof  into two cases.

	\textbf{Case 1.} $m\equiv 2~({\rm mod}\ 3)$.
	
	In this case, since $\frac{a}{b}\in\mathbb{F}_{2^3}$,  we have 
		\begin{equation}\label{4.2}
(\frac{a}{b})^q=(\frac{a}{b})^4, (\frac{a}{b})^{q^2}=(\frac{a}{b})^2.
	\end{equation}
	From  Eq. (\ref{4.1}), we get
   \begin{equation}\label{4.3}
   \frac{a}{b}=(\frac{c}{a})^2=(\frac{a}{b})^{2q^2}.
   \end{equation}
Combining Eqs.  (\ref{4.2}) and    (\ref{4.3}), we have $(\frac{a}{b})^{3}=1$, which contradicts the conclution that $\frac{a}{b}$ is a solution of $x^3+x^2+1=0$.

	\textbf{Case 2.} $m\equiv 0~({\rm mod}\ 3)$.

	In this case, since $\frac{a}{b}\in\mathbb{F}_{2^3}$,  we have 
\begin{equation}\label{4.4}
(\frac{a}{b})^q=\frac{a}{b}.
\end{equation}
Combining Eqs.  (\ref{4.3}) and    (\ref{4.4}), we have $(\frac{a}{b})^{2}=\frac{a}{b}$, which contradicts the facts that $a\neq0$ and $a+b+c=0$.

	Conversely, if $m\equiv 1~({\rm mod}\ 3)$, then $x^3+x^2+1=0$ has three solutions in 
   $\mathbb{F}_{q^3}$ by Lemma \ref{three}. Assume that $\alpha$ is a solution of $x^3+x^2+1=0$ in $\mathbb{F}_{q^3}$. It is clear that $\alpha\in\mathbb{F}_{2^3}\backslash \{1\}$ and $\alpha^q=\alpha^2, \alpha^{q^2}=\alpha^4.$
   Therefore, we have
   \begin{equation*}
   \begin{split}
   f(\alpha)+f(1)&=\alpha+ \alpha^{2q^2+2q}+\alpha^{2q^2+2}+\alpha^{4}+\alpha^{4q}+1\\
   &=\alpha+\alpha^{12}+\alpha^3+\alpha^4+\alpha^8+1\\
   &=\alpha^3+\alpha^4+\alpha^5+1\\
   &=(1+\alpha)^2(1+\alpha^2+\alpha^3)=0,
   \end{split}
   \end{equation*}
 which means that $f(\alpha)=f(1)$. We get a contradiction.
The proof is complete.
\end{proof}

Similarly, we have the following theorem.

\begin{thm}
	Let    $m$ be an  integer with $q=2^m$. Then
	\begin{equation*}
	f(x)=  x^2 + x+x^q+x^{q^2+q}+x^{q^2+1}
	\end{equation*}
	is a permutation polynomial over $\mathbb{F}_{q^3}$ if and only if $m\not\equiv 1~({\rm mod}\ 3)$. 	 	
\end{thm}

\begin{thm}
	Let    $m$ be an  integer with $q=2^m$. Then
	\begin{equation*}
	f(x)=  x^2 + x^{2q^2+2q-2}+x^{2q^3+2q^2-2q}+x^{q^2-q+1}+x^{q^3-q^2+q}
	\end{equation*}
	is a permutation polynomial over $\mathbb{F}_{q^3}$ if and onlu if 	$m\not\equiv 2~({\rm mod}\ 3)$. 	 	
\end{thm}

\begin{proof}

We shall show that for each $d\in\mathbb{F}_{q^3}$, the equation
$f(x)=d$
always has a solution in $\mathbb{F}_{q^3}$.
If $d\in\mathbb{F}_{q}$, then we have $f(d^{\frac{1}{2}})=d$ clearly.

In the following, we shall show that for each $d\in \mathbb{F}_{q^3}\backslash\mathbb{F}_{q}$, the equation
\begin{equation*}
f(x)=  x^2 + x^{3q^2-q+1}+x^{3q^3-q^2+q}+x^{q^2-q+3}+x^{q^3-q^2+3q}=d
\end{equation*} 	
has a solution in  $\mathbb{F}_{q^3}$. We have $x\in \mathbb{F}_{q^3}\backslash\mathbb{F}_{q}$ clearly.	Assume $d=a^2$ for some $a\in \mathbb{F}_{q^3}\backslash\mathbb{F}_{q}$.  
Let $y=x^q,z=y^q,b=a^q,c=b^q$ and $A=a+b+c\in \mathbb{F}_{q}$.  Then  we have the following equations of system,
\begin{align*}
\begin{split}
\left \{
\begin{array}{ll}
x^2+\frac{y^2z^2}{x^2}+\frac{x^2z^2}{y^2}+\frac{xz}{y}+\frac{xy}{z}=a^2,                  \\
y^2+\frac{x^2z^2}{y^2}+\frac{x^2y^2}{z^2}+\frac{xy}{z}+\frac{yz}{x}=b^2,                  \\
z^2+\frac{x^2y^2}{z^2}+\frac{y^2z^2}{x^2}+\frac{yz}{x}+\frac{xz}{y}=c^2.                 \\
\end{array}
\right.
\end{split}
\end{align*}
Adding the  three equations above, we get $x+y+z=a+b+c=A$. Then plugging $z=x+y+A$ into the first two equations of the above system, we have
\begin{align*}
\begin{split}
&f_1(y):x^7 + x^6y + Ax^6 + x^5y + A^2x^5 + A^2x^4y + A^3x^4 + x^3y^4 + a^2x^3y^2 + A^2x^3y+ \\
& x^2y^5 + Ax^2y^4 + a^2x^2y^3 + a^2Ax^2y^2 + xy^6 + A^2xy^4 + y^7 + A y^6+ A^2 y^5+A^3 y^4=0
\end{split}
\end{align*}

and
\begin{equation*}
f_2(y): x^7 + x^3y^4 + b^2x^3y^2 + A^4x^3 + xy^6 + xy^5 + b^2xy^4 + A^2 xy^4+ A^2 xy^3+b^2A^2 xy^2 + y^6 + Ay^5
+ A^2y^4 + A^3y^3=0.
\end{equation*} 	
With the help of Magma, the resultant of $f_1$ and $f_2$
with respect to $y$ is
\begin{equation*}
R(f_1,f_2,y)=x^{21}(x+A)^{12}( \eta_1x^2+\eta_2)^2,
\end{equation*}
where $\eta_1=a^6 + a^4 + a^2b^4 + a^2b^2A^2 + a^2b^2 + a^2A^4 + a^2A^2 +b^6 + b^4A^2+ b^4 + b^2A^2 +A^6 +A^2 +1, $ $\eta_2=a^6 + a^5 A^2 + a^5 + a^4 b^4 + a^4 b^2 A^2 + a^4 b^2 +
a^3 b^4 + a^3 b^2 A^2 + a^3 b^2 + a^3 b A^3 + a^3 b A^2 + a^3 b A + a^3 b + a^3 A^4 + a^3 A^3 + a^3 A + a^2 b^6 
+ a^2 b^4 A^2 + a^2 b^3 A^2 + a^2 b^3 + a^2 b^2 A^3 + a^2 b^2 A + a^2 b^2 + a^2 b A^3 + a^2 b A + a^2 A^6 +
a^2 A^5 + a^2 A^2 + a^2 A + a b^6 + a b^5 A + a b^5 + a b^4 A^2 + a b^4 A + a b^2 A^2 + a b A^3 + a b A^2 +
a A^4 + a A^3 + b^7 + b^6 A + b^6 + b^5 A + b^4 A^4 + b^4 A^3 + b^4 A + b^3 A^2 + b^2 A^3 + b^2 A^2 + b A^3 +
A^8 + A^5 + A^3 + A^2.$
Since $f_1$ and $f_2$ have a common root,
we have
\begin{equation*}
x^{21}(x+A)^{12}( \eta_1x^2+\eta_2)^2=0.
\end{equation*}	

  We first show that $\eta_1\eta_2\neq 0$. Assume that   $u$ is a solution of $x^3+x+1=0$ in  $\mathbb{F}_{q^3}$. It is easy to check that $\eta_1=R_1^2=(a+ub+u^5c+u^2)^2(a+u^2b+u^3c+u^4)^2(a+u^4b+u^6c+u)^2$,
  $\eta_2=S_1S_2$, where $R_1=a^3 + a^2c + ab^2 + abc + ab + ac + a + b^3 + bc^2 + bc + b + c^3 + c + 1$, 
  $S_1=a^3 + a^2c^2 + a^2c + a^2 + ab^2 + abc + ab + ac + a + b^4 + b^3 + bc^2 + bc + b + c^4 + c^3 + c^2 + c,
  S_2=a^4 + a^3 + a^2b^2 + a^2c + a^2 + ab^2 + abc + ab + ac + a + b^3 + b^2 + bc^2 + bc + b + c^4 + c^3 + c.$ It is clear that $S_1^q=S_2.$
  
   The following proof is divided   into two cases.

\textbf{Case 1.} $m\equiv 1~({\rm mod}\ 3)$.

In this case, we have $u^q=u^2$.

\textbf{Case 1-1.}
 Assume  $\eta_1=0$.    Without loss of generality, assume that $a+ub+u^5c+u^2=0$.
Raising  it  to its $q$-th and $q^2$-th power, we get the following equations of system,
\begin{numcases}{}
a+ub+u^5c+u^2=0,        \label{6.1} \\
b+u^2c+u^3a+u^4=0,\label{6.2}\\
c+u^4a+u^6b+u=0.    \label{6.3}
\end{numcases}
By computing (\ref{6.1})$*u^2$+(\ref{6.2})$*u^5$, we have
\begin{equation}\label{6.4}
(u^2+u)a+(u^3+u^5)b+u^4+u^2=0.
\end{equation}	
By computing (\ref{6.3})+(\ref{6.2})$*u^5$, we get 
\begin{equation}\label{6.5}
(u^4+u)a+(u^5+u^6)b+u^2+u=0.
\end{equation}
Combining Eqs. (\ref{6.4}) and (\ref{6.5}), we have $u^2(u^3+u^2+1)b=u^4+u+1.$
Since $u^3+u+1=0$, we have $u^3+u^2+1\neq0$ and  $b=\frac{u^4+u+1}{u^2(u^3+u^2+1)}=1$, which is a contradiction. 

Similarly, we have  $(a+u^2b+u^3c+u^4)(a+u^4b+u^6c+u)\neq 0$, which implies $\eta_1\neq 0$.

\textbf{Case 1-2.}
Assume  $\eta_2=0$. Then we have $S_1=S_2=0.$
Therefore, $S_1+S_2=(a^2+ab+ac+b^2+b+c)^2=0$. Then we have the following equations of system,
\begin{numcases}{}
g_1(c): a^2+ab+ac+b^2+b+c=0,       \label{6.6} \\
g_2(c): b^2+bc+ab+c^2+c+a=0.  \label{6.7}
\end{numcases}
Since $g_1$ and $g_2$ have a common root,
we have the resultant of $g_1$ and $g_2$ with respect to $y$ is 
\begin{equation*}
R(g_1,g_2,c)=(a+b)(a+u^3b+u)(a+u^5b+u^4)(a+u^6b+u^2)=0.
\end{equation*}
If $a=b$, then we have $a\in\mathbb{F}_{q}$, which is a contradiction.
If $a+u^3b+u=0$, then  we have $a=u^3b+u$, $c=ub+u^3$. Plugging the above two equations into Eq.(\ref{6.6}), we get $b=1$, which is a contradiction.

Similarly, we have  $(a+u^5b+u^4)(a+u^6b+u^2)\neq 0$, which implies $\eta_2\neq 0$.

\textbf{Case 2.} $m\equiv 0~({\rm mod}\ 3)$.

In this case, we have $u^q=u$.

\textbf{Case 2-1.}
Assume  $\eta_1=0$.    Assume that $a+ub+u^5c+u^2=0$.
Raising  it  to its $q$-th and $q^2$-th power, we get the following equations of system,
\begin{numcases}{}
a+ub+u^5c+u^2=0,        \label{6.8} \\
b+uc+u^5a+u^2=0,\label{6.9}\\
c+ua+u^5b+u^2=0.    \label{6.10}
\end{numcases}
By computing (\ref{6.8})+(\ref{6.9})$*u^2$, we have
\begin{equation}\label{6.11}
(u^2+u)b+(u^3+u^5)c+u^4+u^2=0.
\end{equation}	
By computing (\ref{6.8})+(\ref{6.11})$*u^6$, we get 
\begin{equation}\label{6.12}
(u^4+u)b+(u^5+u^6)c+u^2+u=0.
\end{equation}
Combining Eqs. (\ref{6.11}) and (\ref{6.12}), we have   $b=\frac{u^4+u+1}{u^2(u^3+u^2+1)}=1$, which is a contradiction. 

Similarly, we have  $(a+u^2b+u^3c+u^4)(a+u^4b+u^6c+u)\neq 0$, which implies $\eta_1\neq 0$.

\textbf{Case 2-2.}
Assume  $\eta_2=0$. 
Then $S_1+S_2=(a^2+ab+ac+b^2+b+c)^2=0$. Therefore, we have the following equations of system,
\begin{numcases}{}
g_1(c): a^2+ab+ac+b^2+b+c=0,       \label{6.13} \\
g_2(c): b^2+bc+ab+c^2+c+a=0.  \label{6.14}
\end{numcases}
We have the resultant of $g_1$ and $g_2$ with respect to $y$,
\begin{equation*}
R(g_1,g_2,c)=(a+b)(a+u^3b+u)(a+u^5b+u^4)(a+u^6b+u^2)=0.
\end{equation*}
If $a=b$, then we have $a\in\mathbb{F}_{q}$, which is a contradiction.
If $a+u^3b+u=0$, then  we have 
$(a+u^3b+u)+u^3(a+u^3b+u)^q=a+u^6c+u+u^4=0$ and  $u^6(a+u^3b+u)^{q^2}=u^2a+u^6c+1=0$.
Combining the above two equations, we get $a=\frac{u^4+u+1}{u^2+1}=1$, which is a contradiction.  

Similarly, we have  $(a+u^5b+u^4)(a+u^6b+u^2)\neq 0$, which implies $\eta_2\neq 0$.

Recall that $\eta_1=R_1^2$,
$\eta_2=S_1S_2$, where $S_2=S_1^q.$ Let $R_2=R_1^q, R_3=R_1^{q^2}, S_3=S_1^{q^2}$.  
Let $ x=(\frac{S_1S_2}{R_1^2})^{\frac{1}{2}}$. 
Then $x^2=\frac{S_1S_2}{R_1^2}, y^2=\frac{S_2S_3}{R_2^2}, z^2=\frac{S_1S_3}{R_3^2}$.  
In the following, it suffices to show that
\begin{equation*}
\begin{split}
f(x)&=x^2+\frac{y^2z^2}{x^2}+\frac{x^2z^2}{y^2}+\frac{xz}{y}+\frac{xy}{z}\\
&=\frac{S_1S_2}{R_1^2}+\frac{S_3^2R_1^2}{R_2^2R_3^2}+\frac{S_1^2R_2^2}{R_1^2R_3^2}+\frac{S_1R_2}{R_1R_3}+\frac{S_2R_3}{R_2R_3}\\
&=a^2,
\end{split}
\end{equation*}
which is equivalent to
$$S_1S_2R_2^2R_3^2+S_3^2R_1^4+S_1^2R_2^4+S_1R_1R_2^3R_3+S_2R_1R_2R_3^3=a^2R_1^2R_2^2R_3^2.$$
By Magma, it is easy to check that the above equation holds clearly.

	Conversely, if $m\equiv 2~({\rm mod}\ 3)$, then $x^3+x^2+1=0$ has three solutions in 
$\mathbb{F}_{q^3}$ by Lemma \ref{three}. Assume that $\alpha$ is a solution of $x^3+x^2+1=0$ in $\mathbb{F}_{q^3}$. It is clear that $\alpha\in\mathbb{F}_{2^3}\backslash \{1\}$ and $\alpha^q=\alpha^4, \alpha^{q^2}=\alpha^2.$
Therefore, we have
\begin{equation*}
\begin{split}
f(\alpha)&=\alpha^2+\alpha^3+\alpha^5+\alpha^6+\alpha^3\\
&=\alpha^2(1+\alpha^3+\alpha^4)\\
&=\alpha^2(1+\alpha)\\
&=1=f(1).
\end{split}
\end{equation*}
 We get a contradiction.
The proof is complete.

\end{proof}

\begin{thm}\label{6}
	Let    $m$ be an  integer with $q=2^m$. Then
	\begin{equation*}
	f(x)=  x^2 + x+x^q+x^{2q^2+2q}+x^{2q^2+2}
	\end{equation*}
	is a permutation polynomial over $\mathbb{F}_{q^3}$. 	
\end{thm}
\begin{proof}

We shall show that for any fixed $a\in \mathbb{F}^*_{q^3}$, the   equation $f(x+a)=f(x)$ has no solution in $\mathbb{F}_{q^3}$.
Let $y=x^q,z=x^{q^2},b=a^q,c=a^{q^2}$. Then we have the following equations of system, 
\begin{align*}
\begin{split}
\left \{
\begin{array}{ll}
(x+a)^2 + (x+a)+(y+b)+(y+b)^2(z+c)^2+(x+a)^2(z+c)^2=x^2 + x+y+y^2z^2+x^2z^2,                   \\
(y+b)^2 + (y+b)+(z+c)+(z+c)^2(x+a)^2+(y+b)^2(x+a)^2=y^2 + y+z+z^2x^2+y^2x^2,                 \\
(z+c)^2 + (z+c)+(x+a)+(x+a)^2(y+b)^2+(z+c)^2(y+b)^2=z^2 + z+x+x^2y^2+z^2y^2,                  \\
\end{array}
\right.
\end{split}
\end{align*}
which is equivalent to
\begin{align*}
\begin{split}
\left \{
\begin{array}{ll}
c^2x^2+c^2y^2+(a^2+b^2)z^2+a+b+a^2c^2+a^2+b^2c^2=0,                   \\
a^2y^2+a^2z^2+(b^2+c^2)x^2+b+c+b^2a^2+b^2+c^2a^2=0,               \\
b^2z^2+b^2x^2+(c^2+a^2)y^2+c+a+c^2b^2+c^2+a^2b^2=0.                \\
\end{array}
\right.
\end{split}
\end{align*}
Adding the  three equations above, we get	$a+b+c=0.$ Plugging $c=a+b$ into the last two equations of the above system, we have 
$$x^2+y^2+z^2=\frac{a+b+a^2c^2+a^2+b^2c^2}{c^2}=\frac{b+b^4+a^2}{b^2}.$$
Therefore, we have 
\begin{equation*}
\begin{split}
&ab^2+a^4b^2+b^4+a^2b+a^2b^4+a^2b^2+a^4\\
&= (ab+a+b)(ab+a+b^2)(a^2+ab+b)=0.
\end{split}
\end{equation*}	
Similar to the proofs of Case 1 and Case 2 of Theorem  \ref{1}, we can derive a contradiction. Therefore, $f(x)$ is  a permutation polynomial over $\mathbb{F}_{q^3}$.
\end{proof}

\begin{thm}
	Let    $m$ be an  integer with $q=2^m$. Then
	\begin{equation*}
	f(x)=  x^2 + x^{3q^2-q+1}+x^{3q^3-q^2+q}+x^{q^2-q+3}+x^{q^3-q^2+3q}
	\end{equation*}
	is a permutation polynomial over $\mathbb{F}_{q^3}$. 	
\end{thm}

\begin{proof}

	We shall prove that for each $d\in\mathbb{F}_{q^3}$, the equation
\begin{equation*}
f(x)=  x^2 + x^{3q^2-q+1}+x^{3q^3-q^2+q}+x^{q^2-q+3}+x^{q^3-q^2+3q}=d
\end{equation*}
always has a solution in $\mathbb{F}_{q^3}$.
For any $d\in\mathbb{F}_{q}$, we have $f(d^{\frac{1}{2}})=d$ clearly.

In the following, we shall show that for each $d\in \mathbb{F}_{q^3}\backslash\mathbb{F}_{q}$, the equation
\begin{equation*}
f(x)=  x^2 + x^{3q^2-q+1}+x^{3q^3-q^2+q}+x^{q^2-q+3}+x^{q^3-q^2+3q}=d
\end{equation*} 	
  has a solution in  $\mathbb{F}_{q^3}$. It is clear that $x\in \mathbb{F}_{q^3}\backslash\mathbb{F}_{q}$.	Assume $d=a^2$ for some $a\in \mathbb{F}_{q^3}\backslash\mathbb{F}_{q}$.  
Let $y=x^q,z=y^q,b=a^q,c=b^q$ and $A=a+b+c\in \mathbb{F}_{q}$.  Then  we have the following equations of system,
\begin{align*}
\begin{split}
\left \{
\begin{array}{ll}
x^2+\frac{xz^3}{y}+\frac{x^3y}{z}+\frac{x^3z}{y}+\frac{xy^3}{z}=a^2,                  \\
y^2+\frac{yx^3}{z}+\frac{y^3z}{x}+\frac{y^3x}{z}+\frac{yz^3}{x}=b^2,                  \\
z^2+\frac{zy^3}{x}+\frac{z^3x}{y}+\frac{z^3y}{x}+\frac{zx^3}{y}=c^2.                 \\
\end{array}
\right.
\end{split}
\end{align*}
Adding the  three equations above, we get $x+y+z=a+b+c=A$. Then plugging $z=x+y+A$ into the first two equations of the above system, we have
\begin{equation*}
f_1(y):x^3y + A^2 x^3+ x^2y^2 + Ax^2y + a^2xy + A^4x + a^2 y^2+ a^2Ay=0
\end{equation*}
and
\begin{equation*}
f_2(y): x^2y^2 + b^2 x^2+ xy^3 + Axy^2 + b^2xy + b^2Ax + A^2y^3 +A^4 y=0.
\end{equation*} 	
By   Magma, the resultant of $f_1$ and $f_2$
with respect to $y$ is
\begin{equation*}
R(f_1,f_2,y)=A^2x^3(x+A)^2( (a^2 + b^2 + c^2 + c)(a^2 + b^2 + c^2 + b)x^2+((a+b+c)^3+ac)(a+b+c)^3+ab))^2.
\end{equation*}
Since $f_1$ and $f_2$ have a common root,
we have
\begin{equation*}
A^2x^3(x+A)^2( (a^2 + b^2 + c^2 + c)(a^2 + b^2 + c^2 + b)x^2+((a+b+c)^3+ac)(a+b+c)^3+ab))^2=0.
\end{equation*}	
Let $R_1=(a+b+c)^3+ac, R_2=(a+b+c)^3+ab, R_3=(a+b+c)^3+bc, U_1=(a+b+c)^2+b, U_2=(a+b+c)^2+c, U_3=(a+b+c)^2+a.$ We claim that $U_1\neq 0$. Otherwise, $a^2 + b^2 + c^2 +b=(a^2 + b^2 + c^2 + b)^{q^2}=a^2 + b^2 + c^2 + a=0$, which implies   $a=b\in \mathbb{F}_{q}$, a contradiction.
Similarly, $R_1\neq 0$.
Let $ x=(\frac{R_1R_2}{U_1U_2})^{\frac{1}{2}}$. 
Then $x^2=\frac{R_1R_2}{U_1U_2}, y^2=\frac{R_2R_3}{U_2U_3}, z^2=\frac{R_1R_3}{U_1U_3}$.  
In the following, it suffices to show that
\begin{equation*}
\begin{split}
(f(x))^2&=x^4+\frac{x^2z^6}{y^2}+\frac{x^6y^2}{z^2}+\frac{x^6z^2}{y^2}+\frac{x^2y^6}{z^2}\\
&=\frac{R_1^2R_2^2}{U_1^2U_2^2}+\frac{R_1^4R_3^2}{U_1^4U_3^2}+\frac{R_1^2R_2^4}{U_1^2U_2^4}+\frac{R_1^4R_2^2}{U_1^4U_2^2}+\frac{R_2^4R_3^2}{U_2^4U_2^2}\\
&=a^4,
\end{split}
\end{equation*}
which is equivalent to
$$R_1^2R_2^2U_1^2U_2^2U_3^4+R_1^4R_3^2U_2^4U_3^2+R_1^2R_2^4U_1^2U_3^4+R_1^4R_2^2U_2^2U_3^4+R_2^4R_3^2U_1^4U_3^2=a^4U_1^4U_2^4U_3^4.$$
With the help of Magma, one can check that the above equation holds clearly.
We complete the proof.
\end{proof}

\begin{thm}
	Let    $m$ be an  integer with $q=2^m$. Then
	\begin{equation*}
	f(x)=  x^4 + x^{q^2+4}+x^{4q+1}+x^{4q^2+1}+x^{q+4}
	\end{equation*}
	is a permutation polynomial over $\mathbb{F}_{q^3}$ if and only if $m$ is odd. 	
\end{thm}
\begin{proof}
	We shall show that for any fixed $a\in \mathbb{F}^*_{q^3}$, the   equation $f(x+a)=f(x)$ has no solution in $\mathbb{F}_{q^3}$.
	Let $y=x^q,z=x^{q^2},b=a^q,c=a^{q^2}$. Then we have the following equations of system, 
	\begin{align*}
	\begin{split}
	\left \{
	\begin{array}{ll}
	(x+a)^4+(x+a)^4(z+c)+(x+a)(y+b)^4+(x+a)(z+c)^4+(x+a)^4(y+b)=x^4+x^4z+xy^4+xz^4+x^4y,                   \\
	(y+b)^4+(y+b)^4(x+a)+(y+b)(z+c)^4+(y+b)(x+a)^4+(y+b)^4(z+c)=y^4+y^4x+yz^4+yx^4+y^4z,                 \\
	(z+c)^4+(z+c)^4(y+b)+(z+c)(x+a)^4+(z+c)(y+b)^4+(z+c)^4(x+a)=z^4+z^4y+zx^4+zy^4+z^4x,                  \\
	\end{array}
	\right.
	\end{split}
	\end{align*}
	which is equivalent to
	\begin{align*}
	\begin{split}
	\left \{
	\begin{array}{ll}
	(b+c)x^4+(b^4+c^4)x+ay^4+a^4y+az^4+a^4z+a^4b+a^4c+a^4+ab^4+ac^4=0,                   \\
	(c+a)y^4+(c^4+a^4)y+bz^4+b^4z+bx^4+b^4x+b^4c+b^4a+b^4+bc^4+ba^4=0,               \\
	(a+b)z^4+(a^4+b^4)z+cx^4+c^4x+cy^4+c^4y+c^4a+c^4b+c^4+ca^4+cb^4=0.                \\
	\end{array}
	\right.
	\end{split}
	\end{align*}
	Adding the  three equations above, we get	$a+b+c=0.$ Plugging $c=a+b$ into the fist two equations of the above system, we have 
	\begin{equation*}
	f_1(z):ax^4+a^4x+ay^4+a^4y+az^4+a^4z+a^4=0,
	\end{equation*}
	and
	\begin{equation*}
	f_2(z): by^4+b^4y+bz^4+b^4z+bx^4+b^4x+b^4=0. 
	\end{equation*} 
	With the help of Magma, we have 
	 the resultant of $f_1$ and $f_2$
	with respect to $z$,  
	\begin{equation*}
	R(f_1,f_2,z)=a^4b^4(a+b)^4(a^2+ab+b^2)^4=0.
	\end{equation*}
	Since $a+b+c=0$ and $a\neq0$, we have $a\neq b$. Therefore, $a^2+ab+b^2=0$. Then $\frac{a}{b}$ is a solution of $x^2+x+1=0$ in 	$\mathbb{F}_{q^3}$, which contradicts the assumption that $m$ is odd and Lemma \ref{two}.
		
	Conversely, if $m$ is even, since $g(x)=x^3$ is not a permutation polynomial over $\mathbb{F}_{q}$, there exists an element $\alpha\in\mathbb{F}^*_{q}$ such that $x^3+\alpha=0$ has  no solutions in $\mathbb{F}_{q}$. By Lemma \ref{four}, we have that 
	$x^4+\alpha x+\alpha=0$ has only one solution in $\mathbb{F}_{q}$, denoted as $\gamma_0$.
	Then we have $\alpha=\frac{\gamma_0^4}{\gamma_0+1}$.

	  Assume that $\omega$ is a primitive 	element of  ${\mathbb F}_{q^3}$  and $\alpha=\omega^{(q^2+q+1)t}$    for some positive integer $t$. Noting that $3\rvert q^2+q+1$, it is easy to check that $x^3+\alpha=0$ has three solutions in  $\mathbb{F}_{q^3}$. By  Lemma \ref{four}, we have that 
	$x^4+\alpha x+\alpha=0$ has four solutions in $\mathbb{F}_{q^3}$. Assume that $\gamma_1\in\mathbb{F}_{q^3}\backslash\mathbb{F}_{q}$  is a solution  of $x^4+\alpha x+\alpha=0$, which means that $\gamma_1^4+\alpha\gamma_1+\alpha=\gamma_1^4+\frac{\gamma_0^4}{\gamma_0+1}\gamma_1+\frac{\gamma_0^4}{\gamma_0+1}=0$. 
	It is clear that $\gamma_1+\gamma_1^q+\gamma_1^{q^2}\in\mathbb{F}_{q}$ is a solution of $x^4+\alpha x+\alpha=0$. Therefore, $\gamma_1+\gamma_1^q+\gamma_1^{q^2}=\gamma_0$.
	Then we have the following equations,
	\begin{equation*}
	\begin{split}
	f(\gamma_1)&=\gamma_1^4+ \gamma_1^{q^2+4}+\gamma_1^{4q+1}+\gamma_1^{4q^2+1}+\gamma_1^{q+4}\\
	&=\gamma_1^4+\gamma_1(\gamma_1^q+\gamma_1)^{4q}+\gamma_1^4(\gamma_1^q+\gamma_1)^q\\
	&=\gamma_1^4+\gamma_1(\gamma_1^{q^2}+\gamma_0)^{4q}+\gamma_1^4(\gamma_1^{q^2}+\gamma_0)^q\\
	&=\gamma_1^4+\gamma_1\gamma_0^4+\gamma_1^4\gamma_0\\
	&=(\gamma_0+1)(\gamma_1^4+\frac{\gamma_0^4}{\gamma_0+1}\gamma_1)\\
	&=(\gamma_0+1)(\frac{\gamma_0^4}{\gamma_0+1})\\
	&=\gamma_0^4=f(\gamma_0),
	\end{split}
	\end{equation*}
	a contradiction.
	We complete the proof.
\end{proof}

Similarly, we have the following theorem.

\begin{thm}
	Let    $m$ be an  integer with $q=2^m$. Then
	\begin{equation*}
	f(x)=  x^2 + x^{q+2}+x^{q^2+2q}+x^{3q^2}+x^{3}
	\end{equation*}
	is a permutation polynomial over $\mathbb{F}_{q^3}$ if and only if $m\not\equiv 2~({\rm mod}\ 3)$. 	 
\end{thm}

\begin{thm}
	Let    $m, k, i$ be  positive  integers with $q=2^m$. Then
	\begin{equation*}
	f(x)=  x^{2^k+1} + x^{(2^k+1)q}+x^{(2^k+1)q^2}+x^{2^i}+x^{2^iq}
	\end{equation*}
	is a permutation polynomial over $\mathbb{F}_{q^3}$ if and only if 	${\rm gcd}(2^k+1,q-1)=1$.
\end{thm}

\begin{proof}
 Assume that ${\rm gcd}(2^k+1,q-1)=1$. We shall give the proof of the sufficiency.   For each fixed  $d\in \mathbb{F}_{q^3}$, it suffices to prove that  the following equation
\begin{equation*}
x^{2^k+1} + x^{(2^k+1)q}+x^{(2^k+1)q^2}+x^{2^i}+x^{2^i}=d
\end{equation*} 	
has at most one solution in  $\mathbb{F}_{q^3}$.
Assume that $d=a^{2^i}$.
Let $y=x^q,z=y^q,b=a^q,c=b^q$.  Then we have the following equations of system, 
\begin{align}\label{10.1}
\begin{split}
\left \{
\begin{array}{ll}
x^{2^k+1} + y^{2^k+1}+z^{2^k+1}+x^{2^i}+y^{2^i}=a^{2^i},                   \\
x^{2^k+1} + y^{2^k+1}+z^{2^k+1}+y^{2^i}+z^{2^i}=b^{2^i},                  \\
x^{2^k+1} + y^{2^k+1}+z^{2^k+1}+x^{2^i}+z^{2^i}=c^{2^i}.               \\
\end{array}
\right.
\end{split}
\end{align}
From the above  system, we get $x+y=b+c, x+z=a+b, y+z=a+c $. Plugging the above three equations  into the first   equation of System  (\ref{10.1}), we can obtain
\begin{equation*}
(x+a+c)^{2^k+1}=(a+c)^{2^k+1}+(a+b)^{2^k+1}+(b+c)^{2^k+1}+a^{2^i}+b^{2^i}+c^{2^i}.
\end{equation*} 		
Since  		${\rm gcd}(2^k+1,q-1)=1$, we have that ${\rm gcd}(2^k+1,q^3-1)=1$ by Lemma 5 of \cite{KK}. Therefore,  $x=((a+c)^{2^k+1}+(a+b)^{2^k+1}+(b+c)^{2^k+1}+a^{2^i}+b^{2^i}+c^{2^i})^{\frac{1}{2^k+1}}+a+c$, which means that $x^{2^k+1} + x^{(2^k+1)q}+x^{(2^k+1)q^2}+x^{2^i}+x^{2^i}=d$ has at most one solution in  $\mathbb{F}_{q^3}$.

	Conversely, assume that ${\rm gcd}(2^k+1,q-1)=c\neq1$ and $\omega$ is a primitive 	element of  ${\mathbb F}_{q^3}$.   
	 Let $\alpha=\omega^{\frac{q^3-1}{c}}$. It is clear that
	 $\alpha\neq1, $ and $\alpha^{2^k+1}=1, \alpha^{q-1}=1$. Then we have $f(\alpha)=1=f(1)$ clearly, which is a contradiction.
	 This complete the proof.
\end{proof}

\section{QM equivalence}\label{section 7}

In this section, we  compare our results with those constructed in the previous works.

\begin{defn}(\cite{Wu})\label{QM} Two permutation polynomials $f(x)$ and $g(x)$ over $ \mathbb{F}_{p^n} $ are called quasi-multiplicative {\rm (QM)} equivalent if there is an integer $1\leq e\leq p^n-1$ with {\rm gcd}$(e,p^n-1)=1$ and
	$f(x)=\alpha g(\gamma x^e),$ where $p$ is a prime, $n$ is a positive integer and $\alpha,\gamma\in \mathbb{F}_{p^n}^{*}.$
	
\end{defn}

Note that two  {\rm QM} equivalent permutations have the same number of terms.
In Table  \ref{table1},
we list  all known permutation pentanomials over
$ \mathbb{F}_{q^3} ~(q=2^m)$ up to {\rm QM} 
equivalence. 

In order to show the non-QM equivalence of
permutation pentanomials presented in this paper, we use the strategy of \cite{Tu2018} as explained
below:
Let $q = 2^m$ and $f(x)=x^{a_0}+x^{a_1}+x^{a_1q}+x^{a_2}+x^{a_2q}$ be a proposed pentanomials for respective
values of $a_0, a_1, a_2$. Let $g(x) = s_0x^{b_0}+s_1x^{b_1}+s_2x^{b_1q}+s_3x^{b_2}+s_4x^{b_2q}$ be any already known permutation 
pentanomials. To show that $f(x)$ and $g(x)$ are not {\rm QM} equivalent over  ${\mathbb F}_{q^3}$, it is enough to
prove any one of the following two conditions:

\textbf{Condition 1}. Show that there does not exist any integer $d, 1 \leq d \leq q^3-1$, such that
${\rm gcd}(d, q^3-1) = 1$ and $\{da_0, da_1, da_1q, da_2, da_2q\}_{({\rm mod}\ q^3-1)} \equiv\{b_0, b_1, b_1q, b_2, b_2q   \} $,
where the subscript  "(${\rm mod}\ q^3-1$)" represents the set in which each entry
is taken ${\rm mod}\ (q^3-1)$.

\textbf{Condition 2}. If Condition 1 does not satisfy and we get some values of $d$, then compare
the coefficients of $f(x)$ and $c_1g(c_2x^d) $and show that they are not equal for
any $c_1, c_2 \in {\mathbb F}^*_{q^3}$.

Actually, the method in the following  proposition is   effective   to  verify the inequivalence of   our results.

\begin{prop}\label{qm}
	Let  $m>0$ be  integers,  $ q=2^m$. Then the permutation pentanomials
	$f_1(x)=x^2 + x+x^q+x^{2q^2+2q}+x^{2q^2+2}$ (Theorem \ref{6} in this paper )
	is not
	{\rm QM} equivalent to $f_2(x)=x^2+  x^{q^2}+ x+x^{q^2+1}+x^{q+1}$ (N0. 8 in Table (\ref{table1}))
	over $\mathbb{F}_{q^3}$.
\end{prop} 
\begin{proof}
	Suppose $f_1$
	is {\rm QM} equivalent to $f_2$. 	Let  $a_0=2, a_1=1, a_2=q, a_3=2q^2+2q, a_4=2q^2+2$. By Definition \ref{QM}, there exists a 
	positive integer $0<d< q^3-1$ with ${\rm gcd}(d, q^3-1) = 1$ s.t., $\{a_0d, a_1d, a_2d, a_3d, a_4d \}_{({\rm mod}\ q^3-1)} \equiv\{2,  q^2, 1, q^2+1,q+1\} $.
 Assume that  $q^2\equiv a_id ~ ({\rm mod}\ q^3-1), 1\equiv a_jd ~({\rm mod}\ q^3-1),  
	q^2+1\equiv a_kd~({\rm mod}\ q^3-1), q+1\equiv a_ld~({\rm mod}\ q^3-1)$  for $0\leq i,j,k,l\leq 4$.
    Then we have $q^3-1\arrowvert a_j-a_iq$ and $q^3-1\arrowvert a_l-a_kq $. It is clear that $a_i=q, a_j=1, a_k=2q^2+2q, a_l=2q^2+2$ or $a_i=2q^2+2q, a_j=2q^2+2, a_k=q, a_l=1$.  In either case, we can deduce that   $2\equiv 2d ~({\rm mod}\ q^3-1)$, which implies that $d=1$.
	It is impossible.
\end{proof}

Similarly, we can see that the main results presented
in our paper are  {\rm QM} inequivalent to the known permutation pentanomial and 
are   {\rm QM} inequivalent to each other.

\begin{table}[h]
	\centering
	\caption{All   known permutation pentanomials over $ \mathbb{F}_{q^3}~(q=2^m) $ up to {\rm QM} 
		equivalence.   }
	\label{table1}
	\centering
	\scalebox{0.82}{
		\begin{tabular}{cccc}
			\hline
			No. & PPs   & Conditions   \\ \hline
			1& $ \epsilon x +\epsilon x^{q^2}+x^{\theta(q^3+q-q^2)}+x^{\theta(q^2+q-1)}+x^{\theta(q^2-q+1)}$  &$\theta=2^i, i\geq0,m>0$  &\cite{Zhang3}   \\ \hline
			2 &$\epsilon x+\epsilon x^q+x^\theta+x^{\theta q}+x^{\theta q^2}$   &$\theta=2^i, i\geq0,m>0$   &\cite{Zhang3}   \\ \hline
			3& $\epsilon x+\epsilon x^{q}+x^{\theta(q+2)}+x^{\theta(q^2+2q)}+x^{\theta(2q^2+1)}$    &$\theta=2^i, i\geq0,$ $m$ is odd    &\cite{Zhang3}    \\ \hline
			4 &  $\epsilon x+\epsilon x^q+\epsilon x^{q^2}+x^{\theta(3q^2-q-1)}+x^{\theta(3q^3-q^2-q)}$   &$\theta=2^i, i\geq0,m>0$   &\cite{Zhang3}      \\ \hline
			5 &  $\epsilon x^{\theta}+\epsilon x^{\theta q}+\epsilon x^{\theta q^2}+x^{3q^2-q-1}+x^{3q^3-q^2-q}$   &$\theta=2^i, i\geq0,m>0$   &\cite{Zhang3}   \\ \hline
			6 &$x^\theta+x^{\theta q}+ x^{\theta q^2}+\epsilon x^{q^3-q^2+q}+\epsilon x^{q^2+q-1}$   &$\theta=2^i, i\geq0,m>0$       &\cite{Zhang3}   \\ \hline
			7 & $\epsilon x+\epsilon x^q+\epsilon x^{q^2}+ x^\theta+x^{\theta q}$   &$\theta=2^i, i\geq0,m>0$   &\cite{Zhang3}   \\ \hline
			8 & $x^2+\epsilon x^{q^2}+\epsilon x+x^{q^2+1}+x^{q+1}$  &$m$ is odd   &\cite{Zhang3}    \\ \hline
			9 & $\epsilon x+x^2+x^{2q}+x^{q^2+1}+x^{q+1}$   &$m\not\equiv 2~({\rm mod}\ 3)$    &\cite{Zhang3}    \\ \hline
			10 & $\epsilon x+x^2+x^{2q^2}+x^{q^2+1}+x^{q+1}$    &$m\not\equiv 1~({\rm  mod}\  3)$  &\cite{Zhang3}   \\ \hline
			11 &$\epsilon x+x^{2q^2+1}+x^{q+2}+x^{q^2+2}+x^{2q+1}$   &$m>0$  &\cite{Zhang3}  \\ \hline
			12 &$x^4+ \epsilon x^{q^2+1}+\epsilon x^{q+1}+x^{2q^2+2}+x^{2q+2}$   &$m>0$   &\cite{Zhang3}    \\	\hline
			13 &$x^4+ x^{4q}+  x^{4q^2}+\epsilon x^{q+1}+\epsilon x^{q^2+q}$  & $m>0$    &\cite{Zhang3}   \\	\hline
			
			14 &$x^4+ \epsilon x^{q}+ \epsilon  x^{q^2}+x^{2q^2+2}+x^{2q+2}$  &$m>0$  &\cite{Zhang3}  \\	\hline
			
			15 &$x+x^4+x^{4q}+x^{2}+x^{2q}$  &$m\not\equiv 5~({\rm mod}\ 15)$  &\cite{Zhang3}    \\	\hline

			16 &$x^4+x^q +x^{q^2}+x^{2}+x^{2q}$  &$m>0$   &\cite{Zhang3}  \\	\hline
			
			17 &$x^4+x +x^{q}+x^{2}+x^{2q}$  &$m\not\equiv 10~({\rm mod}\ 15)$    &\cite{Zhang3}  \\	\hline

			18 &$x^4+x +x^{q}+x^{2q}+x^{2q^2}$  &$m\not\equiv 10~({\rm mod}\ 15)$   &\cite{Zhang3}   \\	\hline

			19 &$x^4+x^{2q^2+2} +x^{2q+2}+x^{q^2-q+1}+x^{q^3-q^2+q}$   &$m>0$  &\cite{Zhang3}  \\	\hline

		  20 &$x+x^{ q}+ x^{q^2}+a x^{q^2+q-1}+ bx^{q^2-q+1}$   &see Thm. 3.1  &\cite{Wang-Zha}  \\	\hline 
	\end{tabular}}
\end{table}

\section{Conclusion}\label{section 8}

 Let $q=2^m$. In this paper, inspired by the methods in \cite{LiKK}, we further characterize the  
permutation property of the  pentanomials with  the  form  (\ref{0}) over ${\mathbb F}_{q^{3}}~(  L(x)=x+x^q, d_0=1,2,4)$.
With the help of Magma, we have searched for all the  permutation pentanomials with the form 
$x^{d_0}+L(x^{s_2q^2+s_1q+s_0}+x^{t_2q^2+t_1q+t_0})$ over ${\mathbb F}_{q^3}$, where $ L(x)=x+x^q, d_0=1,2,4, s_i, t_j\in[-4,4]$, $0\leq i,j\leq 2$.
We have found   effective methods to solve all  the permutation pentanomials  with the above form and conditions, 
although there   are still  some permutation pentanomials we   discovered  that have not been listed due to their similar proofs.
Based on   the
  numerical results, we propose the following conjecture.

\begin{conjecture}
	Let    $m, k$ be positive  integers with $q=2^m$. Then
\begin{equation*}
f_1(x)=  x^{2^k+1}+x^{2^kq^2+1}+x^{q+2^k}+x^{q^2+2^k}+x^{2^kq+1} 
\end{equation*}
and 
\begin{equation*} 
f_2(x)= x^{2^k+1}+x^{(2^k+1)q}+x^{q+2^k}+x^{q^2+2^k}+x^{2^kq^2+q}
\end{equation*}
are permutation polynomials over $\mathbb{F}_{q^3}$ if and only if ${\rm gcd}(2^k+1,q-1)=1$.

\end{conjecture}

\end{document}